\documentclass[12pt,letterpaper]{amsart}

\allowdisplaybreaks
\usepackage{mathpazo}
\usepackage{amsmath,amsfonts,amssymb}
\usepackage{amsrefs}
\usepackage{amsthm}
\usepackage{latexsym,amsmath,amssymb,amsfonts}
\usepackage{rotating}
\usepackage{mathrsfs}
\usepackage{xypic} \xyoption{all}
\usepackage{amscd}
\usepackage{hyperref} 
\usepackage{euscript}
\usepackage{hhline}
\usepackage{graphicx,epstopdf}
\usepackage{epsfig}
\usepackage{xcolor}
\usepackage{textcomp}
\usepackage[all,color]{xy}

\newlength{\fighskip} \fighskip=2pt
\newlength{\figvskip} \figvskip=3pt

\usepackage{hyperref}
\newcommand*{\figbox}[2]{{
  \def\figscale{#1}
  \def\arraystretch{0.8}
  \arraycolsep=0pt
  \begin{array}{c}
    \vbox{\vskip\figscale\figvskip
      \hbox{\hskip\figscale\fighskip
        \includegraphics[scale=\figscale]{#2}}}
  \end{array}}}

\advance\textheight by \topskip
\numberwithin{equation}{section}

\newcommand{\C}{\mathbb{C}}

\newcommand{\Z}{\mathbb{Z}}

\newcommand{\W}{\mathcal{W}}

\DeclareMathOperator{\End}{\text{End}}

\DeclareMathOperator{\Sym}{\text{Sym}}

\theoremstyle{plain}
\newtheorem{thm}{Theorem}[section]
\newtheorem{thm-defn}{Theorem/Definition}[section]
\newtheorem{lem}[thm]{Lemma}
\newtheorem{question}[thm]{Question}
\newtheorem{lem-defn}[thm]{Lemma/Definition}
\newtheorem{prop}[thm]{Proposition}
\newtheorem{cor}[thm]{Corollary}

\theoremstyle{definition}
\newtheorem{defn}[thm]{Definition}
\newtheorem{notn}[thm]{Notation}
\newtheorem{eg}[thm]{Example}

\theoremstyle{remark}
\newtheorem{rmk}[thm]{Remark}

\allowdisplaybreaks[4]  

\begin{document}
\title[Geometric representations of the BT quantization]{A geometric construction of representations of the Berezin-Toeplitz quantization}

\author[Chan]{Kwokwai Chan}
\address{Department of Mathematics\\ The Chinese University of Hong Kong\\ Shatin\\ Hong Kong}
\email{kwchan@math.cuhk.edu.hk}

\author[Leung]{Naichung Conan Leung}
\address{The Institute of Mathematical Sciences and Department of Mathematics\\ The Chinese University of Hong Kong\\ Shatin \\ Hong Kong}
\email{leung@math.cuhk.edu.hk}

\author[Li]{Qin Li}
\address{Shenzhen Institute for Quantum Science and Engineering\\ Southern University of Science and Technology\\ Shenzhen\\ China}
\email{liqin@sustech.edu.cn}


\thanks{}

\begin{abstract}
	For a K\"ahler manifold $X$ equipped with a prequantum line bundle $L$, we give a geometric construction of a family of representations of the Berezin-Toeplitz deformation quantization algebra $(C^\infty(X)[[\hbar]],\star_{BT})$ parametrized by points $z_0 \in X$. The key idea is to use peak sections to suitably localize the Hilbert spaces $H^{0}\left( X,L^{\otimes m}\right) $ around $z_{0}$ in the large volume limit.  
\end{abstract}

\maketitle


\section{Introduction}

Quantization plays important roles in both physics and in mathematics. Two outstanding approaches are the deformation quantization (\cite{DQ-I, DQ-II, Fed, Bordemann-Meinrenken, Ma-Ma, Kontsevich-DQ}) and geometric quantization \cite{Kirillov, Kostant, Souriau, Sniatycki, Woodhouse},\footnote{These lists of references are certainly not meant to be exhaustive.} which provide mathematical descriptions of the algebra of quantum observables and the Hilbert space associated to a quantum mechanical system respectively. This paper is an attempt to understand the intriguing relationship between these two quantization schemes. More precisely, we will construct Hilbert space representations of deformation quantization algebras. In particular, this gives an answer to an open problem in \cite[Sect. 9]{Bordemann-Waldmann}. 

To begin with,  let us consider a symplectic vector space $X = \mathbb{R}^{2n}$ equipped with the standard symplectic form $\omega =\sum_{j = 1}^n dx_{j}\wedge dy_{j}$. A complex polarization (i.e., a complex structure) identifies $X$ with $\mathbb{C}^n$ with coordinates $z_{j}=x_{j}+\sqrt{-1}y_{j}$'s. Geometric quantization of $\C^n$ gives the {\em Bargmann-Fock space} $\mathcal{H}L^2(\mathbb{C}^n,\mu_\hbar)$ consisting of $L^2$ integrable entire holomorphic functions with respect to the density $\mu_\hbar(z) := (\pi\hbar)^{-n}e^{-|z|^2/\hbar}$.
A smooth function $f = f(z, \bar{z}) \in C^{\infty}(X)$ acts on $\mathcal{H}L^2(\mathbb{C}^n,\mu_\hbar)$ as a Toeplitz operator $T_f$ defined by $T_f=\Pi\circ m_f$, where $\Pi$ is the orthogonal projection of smooth functions to $\mathcal{H}L^2(\C^n,\mu_\hbar)$ and $m_f$ is the multiplication by $f$. Typical examples are $T_{z_j}=m_{z_j}$ and $T_{\bar{z}_j}=\hbar\frac{\partial}{\partial{z_j}}$. Composition of these operators defines a {\em star product} via $T_f\circ T_g=T_{f\star g}$. This endows $C^{\infty}(X)[[\hbar]]$ with a noncommutative algebra structure, or a deformation quantization of $(X,\omega)$, and $\mathcal{H}L^2(\mathbb{C}^n,\mu_\hbar)$ is naturally its representation. An algebraic formulation of this deformation quantization and geometric quantization on $\C^n$ is given by the Wick algebra and its holomorphic Bargmann-Fock representation:
\begin{defn}\label{definition: Wick-algebra}
	The {\em Wick algebra} is $\mathcal{W}_{\mathbb{C}^n}:=\mathbb{C}[[y,\bar{y}]][[\hbar]]$ equipped with the multiplication: 
	\begin{equation}\label{equation: defn-Wick-product}
	f\star g:=\exp\left(-\hbar\sum_{i=1}^n\frac{\partial}{\partial y_i}\frac{\partial}{\partial \bar{y}'_i}\right)(f(y,\bar{y})g(y',\bar{y}'))|_{y=y'}.
	\end{equation}
\end{defn}

On the holomorphic Bargmann-Fock space 
$$\mathcal{F}_{\mathbb{C}^n}:=\mathbb{C}[[y_1,\cdots,y_n]][[\hbar]],$$
the Toeplitz operator associated to a monomial  $y_{i_1}\cdots y_{i_k}\bar{y}_{j_1}\cdots \bar{y}_{j_l}$ is the differential operator acting on $s\in\mathcal{F}_{\mathbb{C}^n}$ as
\begin{equation}\label{equation: representation-Wick-on-Fock}
(y_{i_1}\cdots y_{i_k}\bar{y}_{j_1}\cdots \bar{y}_{j_l})\circledast s:= \left(\hbar\cdot\frac{\partial}{\partial y_{j_1}}\right)\circ\cdots\circ\left(\hbar\cdot\frac{\partial}{\partial y_{j_l}}\right)\circ m_{y_{i_1}\cdots y_{i_k}}(s).
\end{equation}
In other words, holomorphic polynomials are mapped to {\em creation} operators and anti-holomorphic ones are mapped to {\em annihilation} operators. This assignment is also commonly known as the {\em Wick normal ordering}.

We also have the {\em anti-holomorphic} Bargmann-Fock representation of $\W_{\C^n}$ on $\bar{\mathcal{F}}_{\C^n}:=\mathbb{C}[[\bar{y}_1,\cdots,\bar{y}_n]][[\hbar]]$, where the operator associated to a monomial is given by
 \begin{equation}\label{equation: anti-holomorphic-Fock}
 (y_{i_1}\cdots y_{i_k}\bar{y}_{j_1}\cdots \bar{y}_{j_l})\bar{\circledast} s:=  m_{\bar{y}_{j_1}\cdots \bar{y}_{j_l}}\circ \left(-\hbar\cdot\frac{\partial}{\partial \bar{y}_{i_1}}\right)\circ\cdots\circ\left(-\hbar\cdot\frac{\partial}{\partial \bar{y}_{i_k}}\right)(s),
 \end{equation}
namely, holomorphic polynomials are now mapped to annihilation operators while anti-holomorphic ones give creation operators.

In \cite{Bordemann-Waldmann}, Bordemann and Waldmann showed that the representation $\bar{\mathcal{F}}_{\C^n}$ is isomorphic to the {\em GNS representation} of $\mathcal{W}_{\C^n}$. Let us briefly recall the construction of GNS states: the delta functional
$$\delta: \mathcal{W}_{\C^n}\rightarrow\C[[\hbar]]$$
defined by taking the constant term induces the {\em Gelfand ideal} 
$$
\mathcal{J}:=\{\alpha\in\mathcal{W}_{\C^n}: \delta(\bar{\alpha}\star\alpha)=0\}.
$$
Then the quotient $\mathcal{W}_{\C^n}/\mathcal{J}$ is naturally a representation of $\mathcal{W}_{\C^n}$ via left multiplication. It was shown in \cite[Proposition 7]{Bordemann-Waldmann} that $\mathcal{W}_{\C^n}/\mathcal{J}\cong\bar{\mathcal{F}}_{\C^n}$ as representations of $\W_{\C^n}$. 

It is a more interesting and difficult problem to find representations of deformation quantization algebras for general K\"ahler manifolds. Bordemann and Waldmann generalized their construction of the GNS representation (or the anti-holomorphic Bargmann-Fock representation) to an arbitrary K\"ahler manifold $X$ by using their previous construction of Wick type Fedosov star products \cite{Bordemann}. They obtained a family of GNS representations parametrized by points in $X$, and raised the following interesting question which motivates the work in this paper:
\begin{question}[problem iv in Sect. 9 in \cite{Bordemann-Waldmann}]\label{question: GNS-prequantum}
Are the prequantum line bundles of geometric quantization over a compact prequantizable K\"ahler manifold related to the GNS construction?
\end{question}

For the flat space $X=\C^n$, we have seen that the Toeplitz operators and the GNS construction correspond to ``conjugate'' representations of $\mathcal{W}_{\C^n}$ defined in \eqref{equation: representation-Wick-on-Fock} and \eqref{equation: anti-holomorphic-Fock} respectively. Note that these two representations are {\em not isomorphic}, since holomorphic polynomials act on $\mathcal{F}_{\C^n}$ and $\bar{\mathcal{F}}_{\C^n}$ as creators and annihilators respectively.
In this paper, we will see that it is actually the generalization of the {\em holomorphic} Bargmann-Fock representation to prequantizable K\"ahler manifolds which encode the geometry of prequantum line bundles.


On such a K\"ahler manifold $X$, there exists a prequantum line bundle $L$ whose curvature $F_L$ satisfies $\frac{\sqrt{-1}}{2\pi} F_L = \omega$. Geometric quantization of $(X, m\omega)$ produces the Hilbert space $H^0(X, L^{\otimes m})$, the space of holomorphic sections of $L^{\otimes m}$. To a smooth function $f\in C^{\infty}(X)$, we can, as in the flat case, associate the Toeplitz operator
\begin{equation*}
T_{f,m}:=\Pi_m\circ m_f: H^0(X,L^{\otimes m})\rightarrow H^0(X,L^{\otimes m}),
\end{equation*}
where $m_f$ is multiplication by $f$ and $\Pi_m$ is the orthogonal projection from the space of $L^2$ sections $L^2(X,L^{\otimes m})$ to $H^0(X,L^{\otimes m})$. 

An important result in {\em Berezin-Toeplitz quantization} is that this gives rise to a star product $\star_{BT}$, and hence the deformation quantization algebra $(C^\infty(X)[[\hbar]],\star_{BT})$ \cites{Bordemann-Meinrenken, Schlichenmaier, Karabegov}:
$$
f\star_{BT}g := \sum_{i\geq 0}\hbar^iC_i(f,g),
$$
such that the following estimates hold:
\begin{equation}\label{equation: asymptotic-expansion-composition-Toeplitz-operators}
||T_{f,m}\circ T_{g,m}-\sum_{i=0}^{N-1}\left(\frac{1}{m}\right)^iT_{C_i(f,g),m}||\leq K_N(f,g)\left(\frac{1}{m}\right)^N;
\end{equation}
here $C_i(-,-)$ are bi-differential operators, $||\cdot||$ is the operator norm, and $K_N(f,g)$ is a constant independent of $m$.
Unlike the flat case, however, the estimate \eqref{equation: asymptotic-expansion-composition-Toeplitz-operators} says that the difference
$$T_{f, m} \circ T_{g, m} - T_{f\star_{BT} g, m}$$
is only {\em asymptotically} zero when $m$ tends to infinity. So $(C^\infty(X)[[\hbar]],\star_{BT})$ does not quite act on $H^{0}( X,L^{\otimes m})$. In fact we do not even expect a representation of $(C^\infty(X)[[\hbar]],\star_{BT})$ on the product $\prod_{m}H^{0}( X,L^{\otimes m})$.

On the other hand, as $m \to \infty$, we are scaling $X$ to a {\em large volume limit}. Physically speaking, we would expect the physical system around any given point $z_{0}\in X$ to behave like one on a flat space. We are going to see that this is indeed the case. To be more precise, we will use \textit{peak sections} $S_{m,p,r} \in H^{0}\left( X,L^{\otimes m}\right)$ \cite{Tian} to appropriately localize the Hilbert spaces around $z_{0}$ and produce a representation $H_{z_{0}}$ of the Berezin-Toeplitz deformation quantization algebra $(C^\infty(X)[[\hbar]],\star_{BT})$. 

In a suitably chosen coordinates (and frame of $L$) around $z_{0}$, $S_{m,p,r}$ is equal to the monomial $z_{1}^{p_{1}}\cdot \cdot \cdot z_{n}^{p_{n}}$ up to order $2r-1$. Because of the error terms, the peak sections in $\prod_{m}H^{0}( X,L^{\otimes m})$ with a fixed $r$ behave in a compatible way with the actions of the operators $T_{z_j}=m_{z_{j}}$ and $T_{\bar{z}_j} = \hbar\frac{d}{dz_{j}}$ around $z_{0}$ only up to order $2r-1$, which is {\em not} enough to produce a representation of $(C^\infty(X)[[\hbar]],\star_{BT})$. To construct our representation $H_{z_{0}}$, we need to find a clever way to increase the order $r$ of Taylor expansions at $z_{0}$ to infinity when the choices of the peak sections $S_{m,p,r}$ are changing with $r$.  

To achieve this, we consider a sum
$$\sum_{m}\alpha _{m,r}\in \prod_{m} H^{0}(X,L^{\otimes m}), $$
or a {\em double sequence} $\left\{ \alpha_{m,r}\right\} $ of peak sections for {\em different} tensor powers $L^{\otimes m}$'s, for which we are only keeping track of the Taylor expansion around $z_{0}$ {\em up to order $2r-1$}. A key observation is that, using asymptotics of inner products of peak sections, we can show that $\left\{ \alpha_{m,r}\right\} $ defines an element in $\prod_{m}H^{0}(X,L^{\otimes m})$ with more and more terms of their Taylor expansions around $z_{0}$ becoming identical if the following condition holds:
there exists a sequence of complex numbers $\{a_{p,k}\}_{p,k\geq 0}$ such that, for each fixed $r>0$,  we have the estimates
\begin{equation}\label{equation: intro-asymptotic-general-Kahler}
\langle\alpha_{m,r}-\sum_{2k+|p|\leq r}a_{p,k}\cdot\frac{1}{m^k}\cdot S_{m,p,r+1},\quad S_{m,q,r+1}\rangle_m=O\left(\frac{1}{m^{r+1}}\right),
\end{equation}
for any multi-index $q$ with $|q|\leq r$.

We call such $\left\{ \alpha _{m,r}\right\} $ an {\em admissible sequence at $z_{0}$} (see Definition \ref{definition: asymptotic_sequence}). They span a linear subspace $V_{z_{0}}\subset \prod_{r}\left(\prod_{m} H^0(X,L^{\otimes m})\right)$. In fact, the coefficients $\left\{ a_{p,k}\right\} $ of $\left\{ \alpha _{m,r}\right\} $ would already record the
whole Taylor expansions at $z_{0}$, and this defines an equivalence relation $\sim$ on $V_{z_{0}}$. The desired Hilbert space can then be constructed as the sub-quotient
$$
H_{z_{0}} := V_{z_{0}}/\sim.
$$
\begin{thm}[=Theorem \ref{theorem: representation-Berezin-Toeplitz-on-asymptotic-sequences}]\label{theorem: intro_representation}
	The vector space $H_{z_0}$ is a representation of the Berezin-Toeplitz deformation quantization algebra $(C^\infty(X)[[\hbar]],\star_{BT})$. 
\end{thm}
One main technical tool we use is the formal integrals defined via the Feynman-Laplace expansions of oscillatory integrals. This technique was also applied in earlier studies of deformation quantization on K\"ahler manifolds \cites{Reshetikhin-Takhtajan, Karabegov19}. In our construction, the formal integrals arise naturally from the geometry of prequantum line bundles, with which we define an inner product on $H_{z_0}$ (with values in $\C[[\hbar]]$) and show that the action of $C^\infty(X)[[\hbar]]$ on $H_{z_0}$ is exactly by the formal Toeplitz operators with respect to this inner product.  

We will prove that our representation possesses various nice properties, as expected from the physical point of view. First of all, Theorem \ref{proposition: explicit-formula-asymptotic-representation} says that it is {\em local}, namely, for any smooth function $f \in C^\infty(X)$, the action of the Toeplitz operator $T_f$ on $H_{z_0}$ depends only on the infinite jets of $f$ at $z_0$. Also, for every real-valued function $f$, we will show that the operator $T_f$ on $H_{z_0}$ is self-adjoint in Proposition \ref{prop:self-adjoint}. Last but not the least, the representation is irreducible in a suitable sense, as we will see in Theorem \ref{proposition: weak-irreducibility}.

In a sequel to this paper \cite{CLL-PartII}, we extend Fedosov's quantization approach to construct a module sheaf over the sheaf of algebras of smooth formal functions under the Berezin-Toeplitz star product. The Hilbert space $H_{z_0}$ will be seen as a subspace of the stalk of this module sheaf at $z_0$ and we show that the $H_{z_0}$'s are related to each other via parallel transport by a Fedosov flat connection. In particular, the action of a smooth formal function $f$ on $H_{z_0}$ is the fiberwise holomorphic Bargmann-Fock action of the Taylor-Fedosov series of $f$ on $\W_{z_0}:=\widehat{\Sym} (T^*X_{z_0})[[\hbar]]$.

\subsection*{Acknowledgement}
\

We thank Si Li and Siye Wu for useful discussions. The first named author thanks Martin Schlichenmaier and Siye Wu for inviting him to attend the conference GEOQUANT 2019 held in September 2019 in Taiwan, in which he had stimulating and very helpful discussions with both of them as well as Jørgen Ellegaard Andersen, Motohico Mulase, Georgiy Sharygin and Steve Zelditch.

K. Chan was supported by grants of the Hong Kong Research Grants Council (Project No. CUHK14302617 \& CUHK14303019) and direct grants from CUHK.
N. C. Leung was supported by grants of the Hong Kong Research Grants Council (Project No. CUHK14301117 \& CUHK14303518) and direct grants from CUHK.
Q. Li was supported by grants from National Science Foundation of China (Project No. 12071204), and Guangdong Basic and Applied Basic Research Foundation (Project No. 2020A1515011220).

\section{The Feynman-Laplace Theorem and perturbations of the Bargmann-Fock space}\label{section:perturbation-Segal-Bargmann}

In this section, we perform local computations which will be needed for proving Theorem \ref{theorem: intro_representation}. We will first introduce an extension of the Wick algebra by allowing negative powers of $\hbar$, and later consider a perturbation of the holomorphic Bargmann-Fock representation.

Throughout this paper, we will use the following notation for multi-indices: let $I=(i_1,\cdots, i_n)$ and $J=(j_1,\cdots,j_m)$ then we set
$$
y^I := y_1^{i_1}\cdots y_n^{i_n},\quad \bar{y}^J := \bar{y}_1^{j_1}\cdots\bar{y}_n^{j_m}. 
$$
We also use the notations: $|I| := i_1+\cdots+i_n$ and $I! := i_1!\cdots i_n!$.

Let $I,J$ be multi-indices. We assign a $\Z$-grading on $\mathcal{W}_{\mathbb{C}^n}$ by letting the monomial $\hbar^k y^I\bar{y}^J$ to have degree $2k+|I|+|J|$. There is an associated decreasing filtration on $\W_{\mathbb{C}^n}$ given by the set $(\W_{\mathbb{C}^n})_k$ of power series in $\W_{\mathbb{C}}$ whose terms are all of degree $\geq k$. In a similar way, we can define a grading and filtration on both $\mathcal{F}_{\mathbb{C}^n}$ and $\mathbb{C}[[\hbar]]$. Note that this grading is preserved by both the Wick product and the Bargmann-Fock action. 

\begin{defn}\label{defn:extension-Wick-algebra}
	The {\em extended Wick algebra} $\mathcal{W}_{\mathbb{C}^n}^+$ is defined as follows:
	\begin{itemize}
		\item Elements of $\mathcal{W}_{\mathbb{C}^n}^+$ are given by power series, possibly with negative powers of $\hbar$,
		\item The degrees of terms of an element $U\in\mathcal{W}_{\mathbb{C}^n}^+$ have a uniform lower bound which could be negative; equivalently, there exists $k\geq 0$ so that every term in $\hbar^k\cdot U$ has non-negative degree. 
		\item For an element $U\in\mathcal{W}_{\mathbb{C}^n}^+$, there exists a finite number of terms for any given nonnegative total degree.
	\end{itemize}
\end{defn}
\begin{rmk}
	The definition of $\mathcal{W}_{\mathbb{C}^n}^+$ here is {\em different} from the one in \cite[p. 224]{Fed}; in that definition, monomials in $\mathcal{W}_{\mathbb{C}^n}^+$ must have non-negative total degrees. Our extension will be important later for proving Theorem \ref{proposition: weak-irreducibility}. 
\end{rmk}

Note that $\mathcal{W}_{\mathbb{C}^n}^+$ is closed under the Wick product. One reason for considering this extension is to allow the following exponentials:
\begin{eg}
	Let $H\in(\mathcal{W}_{\mathbb{C}^n})_3$, i.e., every term in $H$ is of degree at least $3$, then the following {\em classical} and {\em quantum} exponentials both live in $\mathcal{W}_{\mathbb{C}^n}^+$:
	\begin{align*}
	\exp(H/\hbar) & = 1+\frac{H}{\hbar}+\frac{1}{2!}\frac{H\cdot H}{\hbar^2}+\cdots,\\
	\exp^{\star}(H/\hbar) & = 1+\frac{H}{\hbar}+\frac{1}{2!}\frac{H\star H}{\hbar^2}+\cdots
	\end{align*}
\end{eg}
\begin{notn}
	In this paper, we will use the notation $e^{H/\hbar}$ to denote the classical exponential of $H/\hbar$.  
\end{notn}
We can define an extension $\mathcal{F}_{\C^n}^+$ of the holomorphic Bargmann-Fock space $\mathcal{F}_{\C^n}$ in a similar way. It is clear that there is a natural extended Bargmann-Fock action of $\mathcal{W}_{\C^n}^+$ on $\mathcal{F}_{\C^n}^+$. The following lemma shows that the subspace $\mathcal{F}_{\C^n}\subset\mathcal{F}_{\C^n}^+$ is closed under the action of elements in $\mathcal{W}_{\C^n}^+$ of a special form. 
\begin{lem}\label{lemma: extension-Bargmann-Fock-W-plus}
	Suppose $H=\sum_{k,|I|\geq 0,|J|\geq 0}\hbar^{k} a_{k, I, J} \cdot y^I\bar{y}^J\in(\W_{\mathbb{C}^n})_3$ has no purely holomorphic terms, i.e., $a_{k,I,0}=0$. Then $\mathcal{F}_{\mathbb{C}^n}$ is closed under the action of $\exp(H/\hbar)$ and $\exp^\star(H/\hbar)$.
\end{lem}
\begin{proof}
	Each $\bar{y}^j$ in every monomial of $H/\hbar$ acts as $\hbar\frac{\partial}{\partial y^j}$, and the $\hbar$ in this differential operator will cancel with the $\hbar$ in the denominator. So the output can only have nonnegative powers of $\hbar$.  
\end{proof}

\begin{lem}\label{lem:classical-exp-equals-quantum-exp}
	The classical exponential $\exp(H/\hbar)$, where $H\in(\mathcal{W}_{\mathbb{C}^n})_3$, can also be written as a quantum exponential:
	$$
	\exp(H/\hbar)=\exp^\star(H'/\hbar),
	$$
	with $H'\in\mathcal{W}_{\mathbb{C}^n}^+$. In particular, $\exp(H/\hbar)$ is invertible in $\W_{\mathbb{C}^n}^+$ and
	$$
	(\exp(H/\hbar))^{-1}=\exp^\star(-H'/\hbar).
	$$
\end{lem}
\begin{proof}
	Let $A=\exp(H/\hbar)-1\in(\W_{\mathbb{C}^n})_1$. Then $H'$ is defined via the following formal logarithm with respect to the quantum product $\star$:
	$$
	H'=\sum_{k=1}^\infty\frac{(-1)^{k+1}}{k}A^k,
	$$
	where $A^k=A\star A\star \cdots \star A$ denotes the $k$-th power with respect to the quantum product. The fact that $A\in(\W_{\mathbb{C}^n})_1$ implies that each term of $H'$ is of positive degree, and in each degree there are only finitely many terms in $H'$, i.e., $H'\in\mathcal{W}_{\mathbb{C}^n}^+$.
\end{proof}

\subsection{Formal Hilbert spaces}\label{section:formal_Hilbert_spaces}
\


In the K\"ahler geometry setting, we cannot reduce to the local model of the Bargmann-Fock representation and will need to consider the effect of a non-flat metric. For this purpose, we need the following theorem:
\begin{thm}[Feynman-Laplace]\label{theorem: Feynman-Laplace-Gaussian-integral}
	Let $X$ be a compact $n$-dimensional manifold (possibly with boundary), and let $f$ be a smooth function attaining a unique minimum on $X$ at an interior point $x_0 \in X$, and assume that the Hessian of $f$ is non-degenerate at $x_0$; also, let $\mu=\alpha(x)\cdot e^{g(x)}d^nx$ be a top-degree form. Then the integral 
	$$
	I(\hbar):=\int_X \mu e^{-\frac{1}{\hbar}f(x)}=\int_X \alpha(x)\cdot e^{\frac{-f(x)+\hbar g(x)}{\hbar}}dx_1\cdots dx_n,
	$$
	has the following asymptotic expansion as $\hbar\rightarrow 0^{+}$:
	$$
	I(\hbar)\sim\sum_{k\geq 0}a_k\cdot\hbar^k,
	$$
	where each coefficient $a_k$ is a sum of Feynman weights which depends only on the infinite jets of the functions $f,g$ at the point $x_0$. 
\end{thm}
More explicitly, each $a_k$ is a sum over connected graphs of genus $k$. Recall that the genus of a graph $\gamma$ is the sum of the genera of the vertices in $\gamma$ (in our situation, each vertex has genus either $0$ or $1$, labeled by $f(x)$ and $g(x)$ respectively, since the integrand is $e^{\frac{f(x)+\hbar g(x)}{\hbar}}$), and $k = 1-\chi(\gamma)$ where $\chi(\gamma)$ denotes the Euler characteristic of $\gamma$. The propagator in the Feynman weights is given by the inverse of the Hessian of $f$ at $x_0$. The following picture shows a Feynman graph:
	$$
	\figbox{0.28}{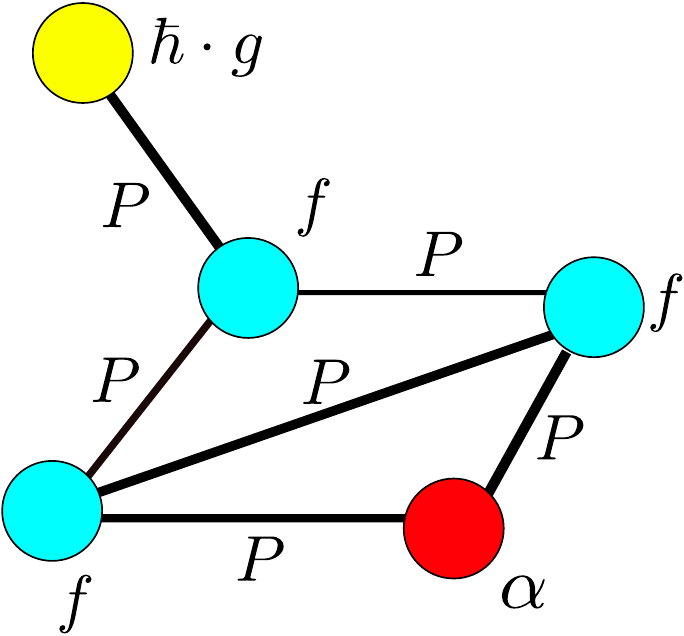}
	$$
	Here every vertex labeled by $f$ must be at least trivalent, and every vertex labeled by $\hbar\cdot g$ must be univalent. 
	For more details on the Feynman-Laplace Theorem, we refer the readers to Pavel Mnev's excellent exposition in \cite{Mnev}, and for a detailed exposition of Feynman graph computations, we refer the readers to \cite{Kevin-book}.

We will mainly apply the Feynman-Laplace Theorem to a function $f(z,\bar{z})$ on a closed disk $\mathbb{D}^{2n} \subset \mathbb{C}^n$ such that the origin $0 \in \mathbb{D}^{2n}$ is the unique minimum of $f$ and $f(0) = 0$.
So, by taking an appropriate complex coordinate system $z=(z_1,\cdots, z_n)$ centered at $0$, the Taylor expansion of $f$ at the origin is given by 
$$
f(z,\bar{z})=|z|^2+O(|z|^3).
$$
Theorem \ref{theorem: Feynman-Laplace-Gaussian-integral} gives an asymptotic expansion of the following integral:
$$
\frac{(\sqrt{-1})^n}{\hbar^n}\int_{\mathbb{D}^{2n}}h(z,\bar{z})e^{\frac{-f(z,\bar{z})+\hbar\cdot g(z,\bar{z})}{\hbar}}dz_1 d\bar{z}_1\cdots dz_n d\bar{z}_n.
$$
\begin{rmk}
	The above integral clearly depends on the radius of $\mathbb{D}^{2n}$, but its asymptotic expansion is actually independent of the radius. 
\end{rmk}

Theorem \ref{theorem: Feynman-Laplace-Gaussian-integral} implies that the asymptotic expansion of the above integral depends only on the Taylor expansions of the functions $f, g$ and $h$ at the origin. We can thus replace these functions by formal power series in $\mathcal{W}_{\mathbb{C}^n}$, and define a formal integral:
\begin{defn}\label{definition: formal_integral}
	For $\phi(y,\bar{y})\in(\W_{\mathbb{C}^n})_3$ and $h(y,\bar{y})\in\W_{\mathbb{C}^n}$, we define the following {\em formal integral}:
	$$
	\frac{1}{\hbar^n}\int h(y,\bar{y})\cdot e^{\frac{-|y|^2+\phi(y,\bar{y})}{\hbar}}\in\mathbb{C}[[\hbar]]
	$$
	via the Feynman rule in Theorem \ref{theorem: Feynman-Laplace-Gaussian-integral}. 
\end{defn}
\begin{rmk}
 	We omit the standard differential form $(\sqrt{-1})^ndz_1 d\bar{z}_1\cdots dz_nd\bar{z}_n$ in the notation of formal integral. 
\end{rmk}

\begin{lem}\label{lemma: formal-integral-preserves-filtration}
 	The formal integral preserves the decreasing filtration on $\W_{\mathbb{C}^n}$ and $\mathbb{C}[[\hbar]]$; more precisely, if $h(y,\bar{y})\in(\W_{\mathbb{C}^n})_k$, then the formal integral lies in $(\mathbb{C}[[\hbar]])_k$.
\end{lem}
\begin{proof}
	The leading term of the formal integral 
	$$
	\frac{1}{\hbar^n}\int h(y,\bar{y})\cdot e^{\frac{-|y|^2+\phi(y,\bar{y})}{\hbar}}\in\mathbb{C}[[\hbar]]
	$$
	is the same as that in the free case, i.e., when $\phi=0$. Thus the leading term of the integral have the same degree as the leading degree of $h(y,\bar{y})$.
\end{proof}

Using this formal integral, we can define a {\em Hilbert space in the formal sense}, namely, its inner product takes values in the formal Laurent series $\mathbb{C}((\sqrt{\hbar}))$:
\begin{defn}
	On the $\mathbb{C}((\sqrt{\hbar}))$-vector space $\mathcal{W}_{\mathbb{C}^n}\otimes_{\mathbb{C}[[\hbar]]}\mathbb{C}((\sqrt{\hbar}))$, we define a {\em complex conjugation} by extending the complex conjugation on polynomials in $\mathbb{C}^n$: 
	$$
	(\sqrt{\hbar})^ka_{I,J}y^I\bar{y}^J\mapsto(\sqrt{\hbar})^k\bar{a}_{I,J}\bar{y}^I y^J.
	$$
	Fix $\phi(y,\bar{y})\in(\W_{\mathbb{C}^n})_3$. Then for $f, g\in\mathcal{W}_{\mathbb{C}^n}((\sqrt{\hbar}))$, we define their {\em formal inner product} as the following formal integral:
	\begin{equation}\label{defn:formal-inner-product}
		\langle f, g\rangle:=\frac{1}{\hbar^n}\cdot\int f\bar{g}\cdot e^{\frac{-|y|^2+\phi(y,\bar{y})}{\hbar}},
	\end{equation}
	which is in turn defined using Feynman graph expansions as in Definition \ref{definition: formal_integral} and takes value in $\mathbb{C}((\sqrt{\hbar}))$. 
\end{defn}

The following are some simple properties of this formal inner product.
\begin{lem}\label{lemma: formal-inner-product-is-hermitian}
	Suppose that $\phi$ is {\em real}, i.e., $\phi=\bar{\phi}$. Then the formal inner product \eqref{defn:formal-inner-product} is Hermitian, namely, $\langle f,g\rangle=\overline{\langle g,f\rangle}$.
\end{lem}
Lemma \ref{lemma: formal-integral-preserves-filtration} implies the following
\begin{cor}
  	The formal inner product of $f,g\in\mathcal{W}_{\mathbb{C}^n}^+$ is a formal power series in $\hbar$, i.e., $\langle f,g\rangle\in\mathbb{C}[[\hbar]]$.
\end{cor}

\begin{rmk}
	We allow $\phi$ to have $\hbar$-dependence. In particular, the fact that $\phi$ is of at least degree $3$ guarantees that the graph expansion of \eqref{defn:formal-inner-product} is valid.
	In the K\"ahler geometry setting, $f$ will be given by the logarithm of the norm of a local holomorphic frame of the prequantum line bundle, and $g$ will be the logarithm of the volume form.
\end{rmk}


The following lemma explains the reason for considering an extension of $\mathcal{W}_{\C^n}$ by $\C((\hbar))$:
\begin{lem}\label{lemma: holomorphic-polynomials-orthonormla-basis}
	The holomorphic polynomials
	$$
	\frac{y^I}{\sqrt{I!\hbar^{|I|}}}
	$$
	form a basis of the formal Hilbert space, which is orthonormal modulo $\hbar$, i.e.,
	$$
	\langle \frac{y^I}{\sqrt{I!\hbar^{|I|}}}, \frac{y^J}{\sqrt{J!\hbar^{|J|}}}\rangle=\delta_{I,J}+O(\hbar).
	$$
\end{lem}
\begin{proof}
	The proof for the cases where $I=J$ is obvious since the computation of the leading term is the same as that in the Bargmann-Fock space. For the cases where $I\not=J$, the terms $y^I\bar{y}^J$ cannot be fully contracted using the quadratic part $-|y|^2/\hbar$. The ``interaction'' part $e^{\phi/\hbar}$ needs to come in so that we can get a full contraction which takes value in $\mathbb{C}[[\hbar]]$. Notice that this contraction preserves the filtration induced by the grading on $\W_{\mathbb{C}^n}$ and $\mathbb{C}[[\hbar]]$. Now  $\frac{y^I}{\sqrt{I!\hbar^{|I|}}}, \frac{\bar{y}^J}{\sqrt{J!\hbar^{|J|}}}$ have degree $0$, and all terms in $e^{\phi/\hbar}$ have degrees strictly greater than $0$. The result follows.  
\end{proof}
\begin{cor}\label{corollary: leading-term-inner-product-polynomial}
	Given two different multi-indices $I\not=J$, we have the following asymptotics:
	$$
	\hbar^{-n}\int y^I\bar{y}^Je^{\frac{-|y|^2+\phi(y,\bar{y})}{\hbar}}=O(\hbar^{\max\{|I|,|J|\}}). 
	$$
	Let $\phi=\sum_{k,I,J}\hbar^k\phi_{k,I,J}y^I\bar{y}^J$, and suppose $\phi$ satisfies the property that $\phi_{0,I,J}=0$
	if either $|I|=1$ or $|J|=1$. Then we further have the refinement:
	$$
	\hbar^{-n}\int y^I\bar{y}^Je^{\frac{-|y|^2+\phi(y,\bar{y})}{\hbar}}=o(\hbar^{\max\{|I|,|J|\}}).
	$$
\end{cor}
\begin{proof}
Let $K=(k_1,\cdots, k_n)$ be the multi-index given by $k_l=\max\{i_l,j_l\}, 1\leq l\leq n$. The worst scenario is when $y^I\bar{y}^J$ together with terms in $e^{\phi/\hbar}$ form a multiple of $y^K\bar{y}^K$ so that we get full contraction. These terms coming from $e^{\phi/\hbar}$ must be a  multiple of
$$
\hbar^{-l}\cdot y^{K-I}\bar{y}^{K-J}
$$
for some $l$. Thus the leading term of the integral in the statement is $O(\hbar^{|K|-l})$. If $l\leq 0$, then there is nothing to show since $|K|-l\geq |K|\geq\max\{|I|,|J|\}$. Thus we assume that $l>0$. 
Since every monomial in $\phi/\hbar$ contains at least one $y^i$'s, it follows that 
\begin{equation}\label{equation: estimate-power-of-hbar}
l\leq|K-I|.
\end{equation}
From Lemma \ref{lemma: holomorphic-polynomials-orthonormla-basis}, we know that
$$
\hbar^{-n}\int y^I\bar{y}^J\cdot \frac{1}{\hbar^l}y^{K-I}\bar{y}^{K-J}=\hbar^{-n}\int\frac{y^K\bar{y}^K}{\hbar^l}=O\left(\hbar^{|K|-l}\right).
$$
The statement follows since $|K|-l\geq|K|-|K-I|=|I|$, and also $|K|-l\geq |J|$ by a similar argument. For the refinement under the additional condition on $\phi$, we only need to notice that the inequality \eqref{equation: estimate-power-of-hbar} can be refined to
$$
l\leq\frac{|K-I|}{2}<|K-I|.
$$
\end{proof}

We want to define the notion of orthogonal projection and formal Toeplitz operators using this formal inner product. To do so, we need the following technical theorem (which is also important in the sequel \cite{CLL-PartII} to this paper):
\begin{thm}\label{theorem: Toeplitz-expression}
	Suppose $\phi\in(\W_{\mathbb{C}^n})_3$ contains no purely holomorphic monomials.  For any  $f\in\mathcal{W}_{\mathbb{C}^n}$, there exists a unique $O_f\in\mathcal{W}_{\mathbb{C}^n}$ such that
	\begin{enumerate}
  		\item For any $s\in\mathcal{F}_{\mathbb{C}^n}$, the element  $T_{O_f}(s)\in\mathcal{F}_{\mathbb{C}^n}$ satisfies the following equalities:
    	$$
    	\langle T_{O_f}(s), y^I \rangle=\langle f\cdot s, y^I\rangle,
    	$$
    	for every multi-index $I$; here $T_{O_f}$ denotes the Bargmann-Fock action by $O_f$.
  		\item IF $f$ is a monomial, then the leading term of $O_f$ is exactly $f$, i.e.,
        $$
        O_f=f+\textit{higher order terms}.
        $$
 	\end{enumerate}
\end{thm}
\begin{proof}
Given $s\in\mathcal{F}_{\mathbb{C}^n}$, suppose there exists $s'\in\mathcal{F}_{\mathbb{C}^n}$ such that 
\begin{equation}\label{equation: equivalent-defn-orthogonal-projection}
 T_{f\cdot e^{\phi/\hbar}}(s)=T_{e^{\phi/\hbar}}(s').
\end{equation}
There is the following straightforward computation:
\begin{align*}
 \int\left(e^{\phi/\hbar}\cdot s'\right)\cdot \bar{y}^I\cdot e^{-\frac{|y|^2}{\hbar}}
 & = \int T_{e^{\phi/\hbar}}(s')\cdot \bar{y}^I\cdot e^{-\frac{|y|^2}{\hbar}}\\
 & = \int T_{f\cdot e^{\phi/\hbar}}(s)\cdot\bar{y}^I\cdot e^{-\frac{|y|^2}{\hbar}}\\
 & = \int\left(f\cdot e^{\phi/\hbar}\cdot s\right)\cdot\bar{y}^I\cdot e^{-\frac{|y|^2}{\hbar}}
\end{align*}
for any purely antiholomorphic monomial $\bar{y}^I$;
here the first equality follows from the fact that $T_{e^{\phi/\hbar}}(s')$ is the orthogonal projection of $e^{\phi/\hbar}\cdot s'$ with respect to the standard Gaussian measure. 
So for every multi-index $I$, we have
$$
\int s'\cdot \bar{y}^I\cdot e^{\frac{-|y|^2+\phi(y,\bar{y})}{\hbar}}=\int fs\cdot\bar{y}^I\cdot e^{\frac{-|y|^2+\phi(y,\bar{y})}{\hbar}}.
$$
Hence we only need to solve the equation \eqref{equation: equivalent-defn-orthogonal-projection} for $s'$.

By Lemma \ref{lem:classical-exp-equals-quantum-exp}, $e^{\phi/\hbar}$ is invertible under the Wick product and its inverse is:
$$
\left(e^{\phi/\hbar}\right)^{-1}=\exp^\star\left(-\sum_{k=1}^\infty\frac{(-1)^{k+1}}{k}(e^{\phi/\hbar}-1)^k\right).
$$
From the following expansion
$$
-\sum_{k=1}^\infty\frac{(-1)^{k+1}}{k}\left(e^{\phi/\hbar}-1\right)^k
=-\sum_{k=1}^\infty\frac{(-1)^{k+1}}{k}\left(\frac{\phi}{\hbar}+\frac{1}{2!}\frac{\phi^2}{\hbar^2}+\cdots\right)^k,
$$
it is easy to see that each monomial in the expansion of $\left(e^{\phi/\hbar}\right)^{-1}$ satisfies the following property: in every $\hbar^{-k}$ term, the antiholomorphic components must have degrees at least $k$. By a similar argument as in the proof of Lemma \ref{lemma: extension-Bargmann-Fock-W-plus}, we see that there is a well-defined action of $\left(e^{\phi/\hbar}\right)^{-1}$ on $\mathcal{F}_{\mathbb{C}^n}$. 

Therefore we get the following explicit description of $s'$:
\begin{equation}\label{equation:formula-orthogonal-projection}
s'=T_{(e^{\phi/\hbar})^{-1}}\circ T_{f\cdot e^{\phi/\hbar}} (s).
\end{equation}
The next step is to look at the term $T_{f\cdot e^{\phi/\hbar}}(s)$ more closely. 
According to \eqref{equation: defn-Wick-product}, we have
$$
f\cdot e^{\phi/\hbar}=e^{\phi/\hbar}\star f-\sum_{k\geq 1}\hbar^k C_k\left(e^{\phi/\hbar}, f\right).
$$
Since $C_k(-,-)$ is a bi-differential operator, the term $C_k\left(e^{\phi/\hbar},f\right)$ is still of the form $e^{\phi/\hbar}\cdot g_k(y,\bar{y})$ for some $g_k\in\mathcal{W}_{\mathbb{C}^n}$ and satisfies the condition that
$$
\text{leading degree of}\ g_k-\text{leading degree of}\ f\geq k.
$$ 
Thus by an induction on the degree, this procedure can be iterated, and we can find  $O_f\in\W_{\mathbb{C}^n}$ whose first terms are exactly $f$ such that $f\cdot e^{\phi/\hbar}=e^{\phi/\hbar}\star O_f$. This implies that
\begin{equation}\label{equation: relation-Taylor-expansion-formal-Toeplitz}
T_{f\cdot e^{\phi/\hbar}}(s)=(T_{e^{\phi/\hbar}}\circ T_{O_f})(s).
\end{equation}
In particular, we see that $s'=T_{O_f}(s)$.
\end{proof}

By the first statement of this theorem, we can make the following:
\begin{defn}
 	The {\em orthogonal projection operator}
	\begin{equation}\label{equation: formal-Toeplitz-operator}
		\pi_{\phi}: \mathcal{W}_{\mathbb{C}^n}\rightarrow \mathcal{F}_{\mathbb{C}^n}=\mathbb{C}[[y_1,\cdots, y_n]][[\hbar]].
	\end{equation}
	is defined by requiring that
	$$
	\langle f, y^I\rangle=\langle\pi_{\phi}(f), y^I\rangle
	$$
	for all multi-indices $I$; here $\langle-,-\rangle$ is the inner product  defined by equation \eqref{defn:formal-inner-product}. 
\end{defn}

We can also define the formal Toeplitz operators:
\begin{defn}\label{defn:formal-Toeplitz-operator}
	The {\em formal Toeplitz operator} $T_{\phi,f}$ associated to $f\in\mathcal{W}_{\mathbb{C}^n}$ is defined as the composition of multiplication by $f$ and the projection $\pi_{\phi}$:
	$$
	T_{\phi,f}:=\pi_{\phi}\circ m_{f}.
	$$
\end{defn}

Theorem \ref{theorem: Toeplitz-expression} gives an explicit algorithm to compute $T_{\phi,f}$: we only need to find $O_f\in\W_{\mathbb{C}^n}$ associated to $f$, and then $T_{\phi,f}=T_{O_f}$.  A simple observation is that if $f=f(y)$ is a holomorphic power series, then $T_f$ is simply the multiplication $m_f$ since then $O_f=f$.  


Here we give a description of the adjoint operator of a formal Toeplitz operator:
\begin{lem}\label{lemma: formal-toeplitz-adjoint-operator}
	Suppose $\phi\in\W_{\mathbb{C}^n}$ is real, i.e., $\phi=\bar{\phi}$. Then for any $f\in\mathcal{W}_{\mathbb{C}^n}$, the adjoint of the formal Toeplitz operator $T_{\phi,f}$ is given by $T_{\phi,\bar{f}}$. In particular, $T_{\phi,f}$ is self-adjoint if and only if $f\in\W_{\mathbb{C}^n}$ is real. 
\end{lem}
\begin{proof}
	According to the definition of formal Toeplitz operators, for any elements $s_1,s_2\in\mathcal{F}_{\mathbb{C}^n}$, we have
	\begin{align*}
	\langle T_{\phi,f}(s_1),s_2\rangle
	& = \int T_{\phi,f}(s_1)\cdot\bar{s}_2\cdot e^{\frac{-|y|^2+\phi(y,\bar{y})}{\hbar}}\\
	& = \int f\cdot s_1\cdot\bar{s}_2\cdot e^{\frac{-|y|^2+\phi(y,\bar{y})}{\hbar}}\\
	& = \int s_1\cdot \overline{\bar{f}\cdot s_2}\cdot e^{\frac{-|y|^2+\phi(y,\bar{y})}{\hbar}}\\
	& = \langle s_1, T_{\phi,\bar{f}}(s_2)\rangle.
	\end{align*}
\end{proof}

\subsection{Local asymptotics via formal Hilbert space}\label{subsection: local-asymptotic}
\

Let $\mathbb{D}^{2n}\subset\C^n$ be a ball centered at $0$, with $dvol_{\mathbb{D}^{2n}}:=(\sqrt{-1})^ne^{\psi(z,\bar{z})}dz_1d\bar{z}_1\cdots dz_nd\bar{z}_n$ the volume form. For every smooth function $f$ on $\mathbb{D}^{2n}$, we will let $J_f\in\mathcal{W}_{\C^n}$ denote the Taylor expansion of $f$ at the origin:
$$
J_f:=\sum_{I,J\geq 0}\frac{1}{I!J!}\frac{\partial^{|I|+|J|}f}{\partial z^I\partial \bar{z}^J}(0)y^I\bar{y}^J.
$$
The previous algebraic computations together with the Feynman-Laplace Theorem \ref{theorem: Feynman-Laplace-Gaussian-integral} give the following asymptotics as $\hbar\rightarrow 0^+$:
\begin{thm}\label{theorem: asymptotics-local}
	 Suppose $\varphi(z,\bar{z})$ is a smooth function on $\mathbb{D}^{2n}$ which attains its unique minimum at the origin.  
	Let $f,\varphi,s$ be functions on $\mathbb{D}^{2n}$ such that $\bar{\partial}s=0$, $\varphi$ has a unique minimum at the origin and satisfies 
	\begin{equation}\label{equation: Taylor-expansion-varphi}
	J_{\varphi}=|y|^2+\sum_{I,J\geq 2}\frac{1}{I!J!}\frac{\partial^{|I|+|J|}\varphi}{\partial z^I\partial \bar{z}^J}(0)y^I\bar{y}^J. 
	\end{equation}
	There exist complex numbers $a_{k,I}$ so that for every fixed multi-index $J$, we have the following asymptotics as $\hbar\rightarrow 0$:
	\begin{equation}\label{equation: formula-local-Toeplitz}
		\frac{1}{\hbar^n}\int_{\mathbb{D}^{2n}}\bigg(f\cdot s-\sum_{2k+|I|\leq r} \frac{1}{\hbar^k}a_{k,I}z^I\bigg)\cdot\bar{z}^Je^{-\frac{\varphi(z,\bar{z})}{\hbar}}dvol_{\mathbb{D}^{2n}}=O(\hbar^{r+1}). 
	\end{equation}
	In particular, these $a_{k,I}$'s only depend on the Taylor expansions of $f,s,\varphi$ and $\psi$ at the origin. 
\end{thm}

\begin{proof}
We define a function $\phi=|z|^2-\varphi(z,\bar{z})+\hbar\psi(z,\bar{z})$.
Then we define $a_{k,I}$'s via the following equation:
$$
\sum_{k,|I|\geq 0}a_{k,I}\hbar^k\cdot z^I=T_{(e^{J_{\phi}/\hbar})^{-1}}\circ T_{J_f\cdot e^{J_{\phi}/\hbar}}(J_s).
$$
From Theorem \ref{theorem: Feynman-Laplace-Gaussian-integral},  we have, for any $\bar{z}^J$, the following equality of asymptotic $\hbar$-expansions
$$
\frac{1}{\hbar^n}\cdot\int_{\mathbb{D}^{2n}}f\cdot s\cdot\bar{z}^Je^{-\frac{\varphi(z,\bar{z})}{\hbar}}dvol_{\mathbb{D}^{2n}}
=
\frac{1}{\hbar^n}\int J_f\cdot J_s\cdot\bar{y}^Je^{\frac{-|y|^2+J_{\phi}}{\hbar}}.
$$
On the other hand, there is the following identity by Theorem \ref{theorem: Toeplitz-expression}:
$$
\frac{1}{\hbar^n}\int J_f\cdot J_s\cdot\bar{y}^Je^{\frac{-|y|^2+J_{\phi}}{\hbar}}=\frac{1}{\hbar^n}\int\bigg(\sum_{k,|I|\geq 0}a_{k,I}\hbar^k\cdot y^I\bigg)\cdot\bar{y}^Je^{\frac{-|y|^2+J_{\phi}}{\hbar}}.
$$
Now equation \eqref{equation: formula-local-Toeplitz} follows from Corollary \ref{corollary: leading-term-inner-product-polynomial} since the truncated higher order terms will only contribute to integrals of type $o(\hbar^{r+1})$. 
\end{proof}

\section{Geometric representations of the Berezin-Toeplitz quantization}

In this section, we construct a family of representations of the Berezin-Toeplitz deformation quantization on a K\"ahler manifold $X$ parametrized by its points, and describe some basic properties such as locality and irreducibility.

The organization of this section is as follows:
In Section \ref{subsection: asymptotic-sequences-on-C-n}, we use $\C^n$ as the motivating example to illustrate the idea behind the general definition of admissible sequences. 
In Section \ref{subsection: asymptotic-representation}, we construct representations of the Berezin-Toeplitz deformation quantization using peak sections (whose properties are reviewed in Appendix \ref{section: peak-section}), which reduce the proof of Theorem \ref{theorem: intro_representation} to the local computations in Section \ref{section:perturbation-Segal-Bargmann}.
In Section \ref{subsection: locality-irreducibility-representation}, we prove locality and (modified) irreducibility of our representations.

\subsection{Admissible sequences on $\C^n$}\label{subsection: asymptotic-sequences-on-C-n}
\

Recall that the prequantum line bundle $L$ on $\mathbb{C}^n$ is trivialized by a global holomorphic frame $\mathbb{1}$. The Hermitian inner product of $L^{\otimes m}$ under this trivialization is given by
$$
h^m(\mathbb{1}^{\otimes m},\mathbb{1}^{\otimes m})=e^{-m\cdot|z|^2}.
$$
We would like to define an action of $C^\infty(\C^n)[[\hbar]]$ on $\mathcal{F}_{\C^n}$ such that its restriction to polynomials is exactly the Bargmann-Fock action.  To do this, we first apply asymptotic analysis to give an equivalent description of $\mathcal{F}_{\C^n}$.

First of all, we consider the vector space
$$
V:=\prod_{r\geq 0}\left(\prod_{m\geq 0}H^0(\C^n,L^{\otimes m})\right),
$$
an element of which is a double sequence $\alpha=(\alpha_{m,r})$ with $\alpha_{m,r}\in H^0(\C^n,L^{\otimes m})$.

We also consider the map
\begin{equation}\label{equation: map-Fock-to-admissible-sequence}
F:\mathcal{F}_{\C^n}\rightarrow V,\quad a=\sum_{k,I}a_{k,I}\hbar^k z^I\mapsto \alpha=\{\alpha_{m,r}\},
\end{equation}
defined by setting 
\begin{equation*}
\alpha_{m,r} := \left(\sum_{2k+|I|\leq r}m^{-k}\cdot a_{k,I}z^I \right)\otimes\mathbb{1}^m\in H^0(\C^n,L^{\otimes m}). 
\end{equation*}
If the element in $\mathcal{F}_{\C^n}$ is of the form $\sum_{I}a_Iz^I$, i.e., it does not include $\hbar$, then each $\alpha_{m,r}$ is a holomorphic section of $L^{\otimes m}$ which is a polynomial of degree $\leq r$ truncated from $\alpha$ under the trivialization $\mathbb{1}^{\otimes m}$. For general elements, $\hbar^k$ is mapped to $1/m^k$ in the corresponding components of the double sequence. 

We now consider an action of smooth functions on the image of the map \eqref{equation: map-Fock-to-admissible-sequence}. Let $f$ be any smooth function on $\C^n$. It is clear that the Toeplitz operator $T_{f,m}$ is in general not well-defined since $\C^n$ is noncompact. We apply the asymptotic analysis in Section \ref{subsection: local-asymptotic} to obtain the following 

\begin{prop}\label{proposition: Toeplitz-preserves-admissible-sequences}
Let $f$ be any smooth function on $\C^n$, and let $F(a)=\{\alpha_{m,r}\}$ be defined as above. Then there exists $b=\sum_{k,I}\hbar^kb_{k,I}z^I\in\mathcal{F}_{\C^n}$ with $F(b)=\{\beta_{r,m}\}$, satisfying the following asymptotics as $m\rightarrow\infty$:
\begin{equation}\label{equation: asymptotic-flat-case}
m^n\cdot\int_{\C^n}(f\cdot\alpha_{m,r}-\beta_{m,r} )\cdot\bar{z}^J\cdot e^{-m\cdot|z|^2}=O\left(\frac{1}{m^{r+1}}\right)
\end{equation}
for every fixed $r\geq 0$ and multi-index $J$. In particular, the formal power series $b\in\mathcal{F}_{\C^n}$ is uniquely determined by $\alpha$ and the Taylor expansion of $f$ at the origin $0\in\C^n$.  
\end{prop}

\begin{proof}
Explicitly, we need to prove the following:
$$
m^n\cdot\int_{\C^n}(f\cdot\alpha_{m,r}-\sum_{2k+|I|\leq r}\frac{1}{m^k}\cdot b_{k,I}z^I )\cdot\bar{z}^J\cdot e^{-m\cdot|z|^2}=O\left(\frac{1}{m^{r+1}}\right).
$$
Since $\alpha_{m,r}$'s are truncated from the same formal power series, we see that $\{f\cdot\alpha_{m,r}\}$ has the same property. Thus the result follows from Theorem \ref{theorem: asymptotics-local}. 
\end{proof}
\begin{eg}
Let us consider the simplest example where $n=1$, and let $a=z$. Then the double sequence $\{\alpha_{m,r}\}$ is explicitly defined by
$
\alpha_{m,r} := z\otimes\mathbb{1}^m. 
$
Let $f=\bar{z}$, then a simple computation of the Wick ordering gives
$
\hbar\frac{\partial}{\partial z}(z)=\hbar. 
$
It follows that
$$b_{k,I} = 
\begin{cases}
1, & (k,I)=(1,0),\\
0, &  \textit{otherwise}.
\end{cases}$$
\end{eg}

We can interpret Proposition \ref{proposition: Toeplitz-preserves-admissible-sequences} as follows. Let $T_{f,m}$ denote the Toeplitz operators on $H^0(\C^n,L^{\otimes m})$ associated to the function $f$, then the double sequence $\{T_{f,m}(\alpha_{m,r})\}\in V$ can be ``approximated'' by a vector in the image of $F$. To make this precise, we define the subspace of {\em admissible sequences} $V_0\subset V$:
\begin{defn}\label{definition: admissible-sequence-flat-case}
We call $\alpha=\{\alpha_{m,r}\}\in V$ an {\em admissible sequence} if there exists $b=\sum_{k,I}b_{k,I}\hbar^k z^I\in \mathcal{F}_{\C^n}$ with  $F(b)=\{\beta_{r,m}\}$ such that, for every fixed $r > 0$, we have 
$$
m^n\cdot\int_{\C^n}(\alpha_{m,r}-\beta_{m,r} )\cdot\bar{z}^J\cdot e^{-m\cdot|z|^2}=O\left(\frac{1}{m^{r+1}}\right)
$$
for any multi-index $J$.
\end{defn}



It is clear that there are inclusions $F(V)\subset V_0\subset V$. Furthermore, there is a natural equivalence relation $\sim$ on $V$: If $\alpha_i\in V_0, i=1,2$ are two admissible sequences with $b_i$ their corresponding elements in $\mathcal{F}_{\C^n}$ respectively, then we say $\alpha_1$ is equivalent to $\alpha_2$ if $b_1=b_2\in\mathcal{F}_{\C^n}$.

\subsection{Construction of the representation}\label{subsection: asymptotic-representation}
\

For a general compact K\"ahler manifold $X$ equipped with a prequantum line bundle $L$, we define
$$V := \prod_{r\geq 0}\left(\prod_{m\geq 0} H^0(X,L^{\otimes m})\right)$$
as before.
The situation is more complicated than $\C^n$ because $L$ is non-trivial and there are no obvious global holomorphic sections analoguous to $z^I\otimes\mathbb{1}^m$ on $\C^n$. The best replacement (or approximation) for the polynomial sections on $\C^n$ are given by so-called {\em peak sections}, with which we can define double sequences of holomorphic sections with asymptotic properties similar to equation \eqref{equation: asymptotic-flat-case}:


\begin{defn}\label{definition: asymptotic_sequence}
For every point $z_0\in X$, we fix a set $\{S_{m,p,r}\}$ of normalized peak sections centered at $z_0$, as introduced in Section \ref{section: peak-section}. A sequence of holomorphic sections $\alpha = \{\alpha_{m,r}\in H^0(X,L^{\otimes m})\}$, regarded as an element in $V$, is called an {\em admissible sequence at $z_0$} if it satisfies the following two conditions:
\begin{enumerate}
 	\item For every fixed $r$, the norm of the sequence $\{\alpha_{m,r}\}_{m>0}$ has a uniform bound: $$||\alpha_{m,r}||_m\leq C_r.$$
 	\item There is a sequence of complex numbers $\{a_{p,k}\}_{p,k\geq 0}$ such that, for each fixed $r>0$,  we have
	\begin{equation}\label{equation: asymptotic-general-Kahler}
	\langle\alpha_{m,r}-\sum_{2k+|p|\leq r}a_{p,k}\cdot\frac{1}{m^k}\cdot S_{m,p,r+1},\quad S_{m,q,r+1}\rangle_m=O\left(\frac{1}{m^{r+1}}\right),
	\end{equation}
	for any multi-index $q$ with $|q|\leq r$.
\end{enumerate}
We define the subspace $V_{z_0}\subset V$ as the $\mathbb{C}$-linear span of admissible sequences at $z_0$.
\end{defn}
Equation \eqref{equation: asymptotic-general-Kahler} is the analogue of equation \eqref{equation: asymptotic-flat-case} in the flat case.  According to Lemma \ref{lemma: inner-product-peak-section-asymptotic}, the coefficients $a_{p,k}$ are uniquely determined.

\begin{rmk}
	The index $r$ in admissible sequences corresponds to the degree in the Wick algebra and Bargmann-Fock space. 
\end{rmk}

The complex numbers $\{a_{p,k}\}$ are called the {\em coefficients} of the admissible sequence $\alpha$. Note that they are independent of both the tensor power $m$ and the {\em weight index} $r$. 
The coefficients define a natural equivalence relation $\sim$ on $V_{z_0}$, namely, $\alpha$ is equivalent to $\beta$ (denoted as $\alpha \sim \beta$) if and only if the coefficients of $\alpha-\beta$ are all $0$.

\begin{rmk}
	It follows from this definition that, for each fixed $r$, even if we change finitely many terms of the double sequence $\{\alpha_{m,r}\}$, its equivalence class remains the same (cf. direct limits). 
\end{rmk}

The vector space we would like to construct is then simply the quotient by this equivalence relation:
$$
H_{z_0}:=V_{z_0}/\sim.
$$
It follows from asymptotics of inner products of peak sections $S_{m,p,r}$'s as $m\rightarrow \infty$ and equation \eqref{equation: asymptotic-general-Kahler} that $H_{z_0}$ is a formal Hilbert space.  

\begin{rmk}
	The vector space $H_{z_0}$ is defined as a sub-quotient, instead of just as a linear span of peak sections. This is because in general Toeplitz operators do {\em not} preserve the space of peak sections.
\end{rmk}

We now give some examples of admissible sequences:

\begin{eg}\label{example: asymptotic-sequences}
	Suppose we fix any multi-index $q$. Then we define an admissible sequence $\alpha$ as follows: we let $\alpha_{m,r}:=0$ if $m$ is too small and there is no (normalized) peak section of $L^{\otimes m}$ corresponding to the index $r$, and let $\alpha_{m,r}=S_{m,q,r+1}$ be simply the normalized peak section. It is then easy to see that this is indeed an admissible sequence with coefficients $a_{p,r}=0$ if $(p,r)\not=(q,0)$ and $a_{q,0}=1$.   
\end{eg}

\begin{eg}\label{example: equivalent-asymptotic-sequences}
	We give an example of equivalent admissible sequences. Let $\alpha$ be the admissible sequence as in the previous example. We now construct a sequence also consisting of normalized peak sections similar to $\alpha$ but with a higher order error term. Namely, we let $\beta=\{\beta_{m,r}\}$ be the admissible sequence given by $\beta_{m,r}=S_{m,q,r+2}$. It is easy to show that $\alpha\sim\beta$. 
\end{eg}

It is not difficult to see that the admissible sequences in Example \ref{example: asymptotic-sequences} form a basis of $V_{z_0}$ as a $\C$-vector space, and thus every vector can be written as $\{a_{p,k}\}$.  More precisely, we have the following lemma:
\begin{lem}
	We have the following isomorphism of $\mathbb{C}$-vector spaces:
	\begin{equation}\label{equation: asymptotic-sequence-isomorphic-to-Fock-space}
		H_{z_0}\cong \mathbb{C}[[y_1,\cdots, y_n]][[\hbar]].
	\end{equation}
\end{lem}
\begin{proof}
The above isomorphism is given by 
$
\{a_{p,k}\}\mapsto \sum_{p,k} a_{p,k}\cdot \hbar^ky^p. 
$
\end{proof}
\begin{rmk}
	The isomorphism in the above lemma depends on a choice of $K$-coordinates centered at $z_0$, and thus is unique only up to a $U(n)$-transformation. 
\end{rmk}

Now for every sequence of operators $\{A_m\}_{m\geq 0}$, where $A_m\in \End(H^0(X,L^{\otimes m}))$, we have an obvious action on $V$: 
$$
\{\alpha_{m,r}\}\mapsto\{A_m(\alpha_{m,r})\}.
$$
We apply this to the sequence of Toeplitz operators $\{T_{f,m}\}_{m\geq 0}$ associated to any given smooth function $f\in C^\infty(X)$.

\begin{lem}\label{lemma: Toeplitz-operators-preserve-asymptotic-sequences}
	Suppose that $\alpha=\{\alpha_{m,r}\}$ is an admissible sequence. Then $\{T_{f,m}(\alpha_{m,r})\}$ is also an admissible sequence for any smooth function $f$. 
\end{lem}
\begin{proof}
We consider the sequence $T_f(\alpha):=\{T_{f,m}(\alpha_{m,r})\}_{m,r\geq 0}$. First of all, since the operator norm of $T_{f,m}$ is bounded by $||f||_{\infty}$ for any fixed $r>0$, the sequence $\{T_{f,m}(\alpha_{m,r})\}_{m>0}$ has bounded norm. 

To show that $T_f(\alpha)$ satisfies the second asymptotic property, we split the integral which defines that property into two parts: one inside the disk $\{\rho(z)<1\}$ and the other outside:
\begin{align*}
	& m^n\int_X h^m(T_{f,m}(\alpha_{m,r}),S_{m,p,r+1})\cdot dV_g\\
=	& m^n\int_X h^m(f\cdot\alpha_{m,r},S_{m,p,r+1})\cdot dV_g\\
=	& m^n\int_{X\setminus\{\rho(z)<1\}}h^m(f\cdot\alpha_{m,r},S_{m,p,r+1})\cdot dV_g+m^n\int_{\{\rho(z)<1\}}h^m(f\cdot\alpha_{m,r},S_{m,p,r+1})\cdot dV_g,
\end{align*}
where the first equality follows from the fact that $T_{f,m}(\alpha_{m,r})$ is the orthogonal projection of $f\cdot\alpha_{m,r}$ to the space of holomorphic sections, and that $S_{m,p,r+1}$ are all holomorphic sections.  For the integral outside the disk, we have the following estimate:
\begin{align*}
	 & \Bigg|m^n\int_{X\setminus\{\rho(z)<1\}} h^m(f\cdot\alpha_{m,r}, S_{m,p,r+1})\cdot dV_{g}\Bigg|\\
\leq & m^n\left(\int_{X\setminus\{\rho(z)<1\}} ||f\cdot\alpha_{m,r}||^2_{h^m}\cdot dV_{g}\right)^{1/2}\cdot\left(\int_{X\setminus\{\rho(z)<1\}} ||S_{m,p,r+1}||^2_{h^m}\cdot dV_{g}\right)^{1/2}\\
\leq & ||f||_{\infty}\cdot C_r\cdot O\left(\frac{1}{m^{(2r+2+|p|)/2}}\right)=O\left(\frac{1}{m^{r+1}}\right).
\end{align*}
Here the multi-index $p$ satisfies $|p|\leq r$, and the constant $C_r$ is given by the upper bound of the sequence $\{\alpha_{m,r}\}_{m>0}$. For the second inequality, we have used boundedness of $\{\alpha_{m,r}\}_{m>0}$, and equation \eqref{equation: norm-peak-section-outside-disk}.

Hence it remains to consider the integral inside the disk $\{\rho(z)<1\}$. In this local setting, the computation is the same as the asymptotics of the Gaussian integral on $\mathbb{D}^{2n}$, and also that in the formal Hilbert space. So the statement follows from Theorem \ref{theorem: asymptotics-local}.
\end{proof}

This lemma shows that for any smooth function $f$, the sequence of Toeplitz operators $\{T_{f,m}\}_{m>0}$ gives a well-defined linear operator $T_f: V_{z_0} \to V_{z_0}$, $\alpha \mapsto T_f(\alpha)$.

\begin{lem}\label{lemma: Toeplitz-operators-preserve-equivalence-of-asymototic-sequences}
	Suppose that two admissible sequences are equivalent, i.e., $\alpha\sim\beta$. Then for any smooth function $f$, we have $T_f(\alpha)\sim T_f(\beta)$.
\end{lem}
\begin{proof}
We only need to show that if the coefficients $a_{k,I}$ of $\alpha=\{\alpha_{m,r}\}$ vanish, then $T_f(\alpha)$ has the same property. Notice that the condition $a_{k,I}=0$ is equivalent to the following equalities for all indices $r,q$, as $m\rightarrow\infty$:
$$
 m^n\int_X h^m (\alpha_{m,r},S_{m,q,r+1})\cdot dV_g=O\left(\frac{1}{m^{r+1}}\right).
$$
Similar to the argument of Lemma \ref{lemma: Toeplitz-operators-preserve-asymptotic-sequences}, there is
\begin{align*}
\Bigg|m^n\int_X h^m (T_{f,m}(\alpha_{m,r}),S_{m,q,r+1})\cdot dV_g\Bigg|=&\Bigg|m^n\int_X h^m (f\cdot\alpha_{m,r},S_{m,q,r+1})\cdot dV_g\Bigg|\\
\leq&m^n\cdot ||f||_\infty\Bigg|\int_X h^m (\alpha_{m,r},S_{m,q,r+1})\cdot dV_g\Bigg|\\
=&O\left(\frac{1}{m^{r+1}}\right).
\end{align*}
\end{proof}



Hence, for every smooth function $f$ on $X$, the sequence of Toeplitz operators $\{T_{f,m}\}_{m>0}$ gives a well-defined linear operator $T_f$ on the vector space $H_{z_0}$. We can further extend it to an action of $C^\infty(X)[[\hbar]]$ on $H_{z_0}$ by letting $\hbar^k\cdot f$ act as 
$$
T_{\hbar^k\cdot f}:\{\alpha_{m,r}\}\mapsto\left\{\frac{1}{m^k}\cdot T_{f,m}(\alpha_{m,r})\right\}.
$$

Lemma \ref{lemma: formal-toeplitz-adjoint-operator} implies the following:
\begin{prop}\label{prop:self-adjoint}
	Suppose $f\in C^\infty(X)$ is a real function. Then for every $z_0\in X$, the operator $T_f$ on $H_{z_0}$ is self-adjoint.  
\end{prop}

\begin{lem}\label{lemma: asymptotitc-operator-asymptotic-sequence}
	Let $A=\{A_m\}_{m\geq 0}$ and $B=\{B_m\}_{m\geq 0}$ be two sequences of bounded operators preserving asymptotic sequences and satisfying the condition that 
	\begin{equation}\label{equation: asymptotic-operator}
		||A_m-B_m||=O\left(\frac{1}{m^{k+1}}\right).
	\end{equation}
	Then for any admissible sequence $\alpha$, the two admissible sequences $A(\alpha)$ and $B(\alpha)$ have the same coefficients up to weight $k$. 
\end{lem}
\begin{proof}
Equation \eqref{equation: asymptotic-operator} implies that the operators $\{A_m-B_m\}$ will increase the weight of $\alpha$ by $k+1$. The lemma follows.  
\end{proof}

Here is our main theorem:
\begin{thm}\label{theorem: representation-Berezin-Toeplitz-on-asymptotic-sequences}
	Let $z_0\in X$ be any point. The action of $C^\infty(X)[[\hbar]]$ on the vector space $H_{z_0}$ satisfies the following relation:
	\begin{equation}\label{equation: representation-Toeplitz-operator}
		T_f\circ T_g=T_{fg}+\sum_{k\geq 1}\hbar^k\cdot T_{C_k(f,g)}, \hspace{3mm}f,g\in C^\infty(X),
	\end{equation}
	where $C^k(-,-)$ are the bi-differential operators which appear in the Berezin-Toeplitz quantization. Therefore, $H_{z_0}$ is a representation of the Berezin-Toeplitz deformation quantization algebra $(C^\infty(X)[[\hbar]],\star_{BT})$. 
\end{thm}
\begin{proof}
We first recall the property of Toeplitz operators:
$$
||T_{f,m}\circ T_{g,m}-(T_{fg,m}+\sum_{k=1}^n\frac{1}{m^k}\cdot T_{C_k(f,g),m})||=O\left(\frac{1}{m^{n+1}}\right).
$$
We apply Lemma \ref{lemma: asymptotitc-operator-asymptotic-sequence} by putting $A_m:=T_{f,m}\circ T_{g,m}$ and $B_m:=T_{fg}+\sum_{k=1}^n\frac{1}{m^k}\cdot T_{C_k(f,g),m}$. Then $A(\alpha)$ and $B(\alpha)$ have the same coefficients up to order $n$. The theorem follows by letting $n\rightarrow\infty$.
\end{proof}

\begin{rmk}
The representation and also the isomorphism \eqref{equation: asymptotic-sequence-isomorphic-to-Fock-space} are independent of the choice of the set of peak sections because for every multi-index $p$ and $p'>|p|$, different choices of peak sections only differ by higher order terms. 
\end{rmk}

\subsection{Locality and irreducibility of the representation}\label{subsection: locality-irreducibility-representation}
\

\subsubsection{Locality}
We will give an explicit formula of our representation under the isomorphism \eqref{equation: asymptotic-sequence-isomorphic-to-Fock-space}. Given any $K$-coordinates $(z_1,\cdots, z_n)$ centered at $z_0$, we define $J_{f,z_0}\in\W_{\mathbb{C}^n}$ by
\begin{equation}\label{equation: map-function-to-Wick}
J_{f,z_0}:=\sum_{|I|, |J|\geq 0}\frac{1}{I!J!}\frac{\partial^{|I|+|J|}f}{\partial z^I\bar{z}^J}(z_0)y^I\bar{y}^J,
\end{equation}
where the sum is over all multi-indices.
\begin{thm}\label{proposition: explicit-formula-asymptotic-representation}
	Let $f$ be any smooth function on $X$, and $J_{f,z_0}$ be defined as above. We define $O_{f,z_0}\in\W_{\mathbb{C}^n}$ as the unique solution of the following equation:
	$$
	J_{f,z_0}\cdot e^{\Phi/\hbar}=e^{\Phi/\hbar}\star O_{f,z_0}.
	$$
	Then the action of $T_f$ on $\alpha\in H_{z_0}$ is given by 
	$$
	T_f(\alpha)=O_{f,z_0}\star\alpha.
	$$
	In particular, this implies that the representation $H_{z_0}$ is local in $f\in C^\infty(X)$, i.e., it only depends on the infinite jets of $f$ at $z_0$. 
\end{thm}
\begin{proof}
The representation $H_{z_0}$ depends on an arbitrarily small neighborhood of $z_0$. So the result follows from the computation of the formal Toeplitz operators on $\mathbb{C}^n$ in Theorem \ref{theorem: Toeplitz-expression}.  
\end{proof}
As a straightforward corollary, we have:
\begin{cor}\label{corollary: Toeplitz-star-product-formula-via-asymptotic-representation}
	Let $f,g\in C^\infty(X)$ be smooth functions on $X$. Then
	$$
	O_{f\star_{BT}g,z_0}=O_{f,z_0}\star O_{g,z_0}. 
	$$
\end{cor}
This gives an algorithm for computing $f\star_{BT}g$: for every $z_0\in X$, in order to find $(f\star_{BT}g)(z_0)$, we only need to compute the Wick product $O_{f,z_0}\star O_{g,z_0}$ and then collect all the constant terms in the Wick algebra.

\subsubsection{Irreducibility}
We now consider irreducibility of our representation. The first observation is that the Bargmann-Fock space is {\em not} an irreducible representation of $\W_{\mathbb{C}^n}$: for every $f\in\W_{\mathbb{C}^n}$, we have the invariant subspaces
$$
T_{f}\left((\mathcal{F}_{\mathbb{C}^n})_k\right)\subset(\mathcal{F}_{\mathbb{C}^n})_k.
$$
But this is the {\em only} reason why the representation fails to be irreducible.
So to have a suitable notion of irreducibility, we  make use of the extended algebra $\W_{\mathbb{C}^n}^+$, which allows terms with negative degrees, and the corresponding extension $\mathcal{F}_{\mathbb{C}^n}^+$. It is then quite easy to check that we indeed obtain an irreducible representation. 

We now define an extension of $C^\infty(X)[[\hbar]]$, which is the geometric analogue of $\W_{\mathbb{C}^n}^+$.
\begin{defn}
	For every smooth function $f\in C^\infty(X)$, let $\deg_{z_0}(f)$ be the vanishing order of $f$ at $z_0$. Then let $(C^\infty(X)[[\hbar]])_{z_0}^+$ be the extension of $C^\infty(X)[[\hbar]]$ which consists of formal functions:
	$$
	\sum_{i\in\mathbb{Z}}\hbar^{i}\cdot f_i,
	$$
	where $f_i\in C^\infty(X)$ are smooth functions on $X$ satisfying the conditions that
	\begin{itemize}
		\item	the sum $\deg_{z_0}(f_i)+2i$ has a uniform lower bound for all $i$, and
		\item	for every degree $k$, the following expression is a finite sum:
				$$
					\sum_{2i+\deg_{z_0}(f_i)=k}\hbar^i\cdot f_i.
				$$
	\end{itemize}
\end{defn}
In the same way we can define the extension $H_{z_0}\subset H_{z_0}^+$. It is easy to check that the extension $(C^\infty(X)[[\hbar]])_{z_0}^+$ is closed under the star product $\star_{BT}$, and acts on $H_{z_0}^+$. Furthermore, the map $f\mapsto O_{f,z_0}$ can be extended to 
\begin{equation}\label{equation: map-extended-formal-function-extended-Wick}
(C^\infty(X)[[\hbar]])_{z_0}^+\rightarrow \W_{\mathbb{C}^n}^+.
\end{equation}

\begin{thm}\label{proposition: weak-irreducibility}
	For every $z_0\in X$, the representation $H_{z_0}^+$ of $(C^\infty(X)[[\hbar]])^+_{z_0}$ is irreducible. 
\end{thm}
\begin{proof}
Let $W$ be a sub-representation of $H_{z_0}^+$. We choose any non-zero $a\in W$, which can be written as:
$$
a=\sum_{2i+|I|\geq k}a_{i,I}\hbar^i\cdot y^I.
$$
Since for a local holomorphic function $f$, we have
$$
O_{f,z_0}=J_{f,z_0},
$$
which consists of only creators in $(H_{z_0})^+$, we only need to find $f\in(C^\infty(X)[[\hbar]])^+_{z_0}$ such that $T_{f}(a)=\hbar^l$ for some $l$, and the result will follow.  

We choose a non-zero term in $a$ of leading order $a_{i_0, I_0}\hbar^{i_0}\cdot z^{I_0}, 2i_0+I_0=k$, such that $i_0$ is the least possible. Let $f_0\in (C^\infty(X)[[\hbar]])^+_{z_0}$ be a formal function which is $a_{i_0,I_0}^{-1}\bar{z}^{I_0}$ near $z_0$. So the leading term of the image of $f_0$ under the map \eqref{equation: map-extended-formal-function-extended-Wick} is $a_{i_0, I_0}^{-1}\bar{y}^{I_0}$, and the degree of the function $f_0$ is exactly $|I_0|$.   We have
\begin{align*}
T_{f_0}\bigg(\sum_{2i+|I|\geq k}a_{i,I}\hbar^i\cdot y^I\bigg)= \hbar^{i_0+|I_0|}+\sum_{2i+|I|=2(i_0+|I_0|)+1}b_{i,I}\hbar^i y^{I}
+\text{higher degree terms}.
\end{align*}
The next step is to find a formal function $f_1\in(C^\infty(X)[[\hbar]])^+_{z_0}$, so that 
$$
(T_{f_0}+T_{f_1})\bigg(\sum_{2i+|I|\geq k}a_{i,I}\hbar^i\cdot y^I\bigg)=\hbar^{i_0+|I_0|}+\text{higher degree terms}.
$$
Let $g$ be a formal function which equals $-\frac{1}{\hbar^{i_0+|I_0|}}\sum_{2j+|J|=2(i_0+|I_0|)+1}b_{j,J}\hbar^j z^{J}$ near $z_0$. It is easy to see that the total degree of $g$ is 1, and we have
\begin{align*}
(T_{f_0}+T_{g}\circ T_{f_0})\bigg(\sum_{2i+|I|\geq k}a_{i,I}\hbar^i\cdot y^I\bigg)
=\hbar^{i_0+|I_0|}+\left(\text{terms of degree }\geq 2(i_0+|I_0|)+2)\right).
\end{align*}
Although $T_{g}\circ T_{f_0}=\sum_{i\geq 0}\hbar^{i}T_{C_i(g,f_0)}$ is an infinite sum, but those high enough $\hbar$ will map terms in $a$ to terms of high degree.  More precisely, we can simply let $f_1=\sum_{i=0}^N\hbar^iC_i(g,f_0)$, where $N=i_0+|I_0|-k/2+2$. Then we have
$$
(T_{f_0}+T_{f_1})\bigg(\sum_{2i+|I|\geq k}a_{i,I}\hbar^i\cdot y^I\bigg)
=\hbar^{i_0+|I_0|}+\left(\text{terms of degree }\geq 2(i_0+|I_0|)+2)\right).
$$
From the formula of $\star_{BT}$, it is easy to see that the total degree of $\hbar^iC_i(g,f_0)$ is no less than the sum of the degree of $g_0$ and $f$ which equals $1+|I_0|$. This implies that the degree of terms in $f_1$ is strictly greater than the degree of $f_0$. This procedure can be repeated and we obtain the desired formal function $f=\sum_{i\geq 0}f_{i}\in(C^\infty(X)[[\hbar]])_{z_0}^+$. 
\end{proof}

\appendix

\section{Peak sections in K\"ahler geometry}\label{section: peak-section}

In this appendix, we briefly review of the notion of peak sections, which was introduced in \cite{Tian} and plays an important role in the study of asymptotic expansions of the Bergmann kernel \cite{Zelditch} with applications to balanced embeddings and constant scalar curvature metrics as well as in the theory of the geometric quantization. For our purpose, we need to introduce a normalized version of peak sections and describe their basic properties. We first recall the notion of $K$-coordinates and $K$-frame of the prequantum line bundle $L$. 
\begin{defn}
	Let $e_{L,z_0}$ be a holomorphic frame of the prequantum line bundle $L$ in a neighborhood of a point $z_0\in X$, and let $(z_1,\cdots, z_n)$ be a holomorphic coordinate system centered at $z_0$. Let $\varphi(z):=-\log||e_{L,z_0}||$. We say that $(z_1,\cdots, z_n)$ are {\em $K$-coordinates with $K$-frame} $e_{L,z_0}$ if the Taylor expansion of $\varphi(z)$ at $z_0$ is of the following form:
 	\begin{equation}\label{equation: Taylor-expansion-hermitian-metric}
 		\varphi(z)\sim |z|^2+\sum a_{JK}z^J\bar{z}^K,\hspace{5mm} |J|\geq 2, |K|\geq 2.
 	\end{equation}
\end{defn}
For K\"ahler manifolds with real analytic K\"ahler form, the existence of $K$-coordinates and $K$-frames was shown by Bochner. For K\"ahler manifolds with only smooth K\"ahler form, such a coordinate system and frame do not exist in general; then we may consider a weaker $K$-coordinates and $K$-frames of finite order. But to avoid further technical complications, let us assume that the K\"ahler manifolds in this paper always admit $K$-coordinates and $K$-frames. 

It is obvious that this local holomorphic frame $e_{L,z_0}$ is unique up to a multiplication by a complex number of modulus $1$. In particular, the leading term of the Taylor expansion of $\varphi$ with degree at least $3$ is given by the curvature: 
\begin{equation}\label{equation: logarithm-hermitian-inner-product-modified}
	\varphi(z,\bar{z})=|z|^2+\sum_{i,j,k,l} R_{i\bar{j}k\bar{l}}z_iz_k\bar{z}_j\bar{z}_l+O(|z|^5).
\end{equation}
Also note that equation \eqref{equation: Taylor-expansion-varphi} is satisfied.


\begin{lem}\label{lemma: local-volume-form-purely-holomorphic-derivatives-vanish}
	Suppose the volume form is $(\sqrt{-1})^{n}\cdot e^{\psi(z,\bar{z})}\cdot dz^1\cdots dz^n d\bar{z}^1\cdots d\bar{z}^n=\omega^n$. Then the purely (anti-)holomorphic derivatives of $\psi(z,\bar{z})$ vanish at $z_0$ under the $K$-coordinates, i.e.,
	$$
	\frac{\partial^{|I|}\psi}{\partial z^I}(z_0)=\frac{\partial^{|J|}\psi}{\partial \bar{z}^I}(z_0)=0
	$$
	for all multi-indices with $|I|,|J|>0$. 
\end{lem}
\begin{proof}
We will only prove the vanishing of purely holomorphic derivatives at $z_0$; the proof for antiholomorphic ones is the same. It suffices to show that the statement is valid for functions $\omega_{i_1\bar{j_1}}\cdots\omega_{i_n\bar{j_n}}$, where
$$
\omega_{i\bar{j}}=\frac{\partial^2\varphi}{\partial z^i\partial \bar{z}^j},
$$
and $\phi$ is K\"ahler potential. But then equation \eqref{equation: Taylor-expansion-hermitian-metric} implies that
$$
\frac{\partial^{|I|+1}\varphi}{\partial z^I\partial\bar{z}^j}(z_0)=0,\ |I|\geq 2,
$$
from which the statement follows easily.
\end{proof}
Equation \eqref{equation: Taylor-expansion-hermitian-metric} together with Lemma \ref{lemma: local-volume-form-purely-holomorphic-derivatives-vanish} tell us that the Taylor expansions of $\varphi(z,\bar{z})$ and $\psi(z,\bar{z})$ satisfy the technical conditions in Theorem \ref{theorem: asymptotics-local}. This allows us to apply the algebraic computations in the formal setting in Section \ref{section:formal_Hilbert_spaces}. 

We now recall the following proposition in Tian's paper \cite{Tian} and, in particular, the definition of peak sections. Let $(z_1,\cdots, z_n)$ be a $K$-coordinate with $K$-frame $e_{L,z_0}$ at $z_0\in X$, and consider the function $\rho(z) := \sqrt{|z_1|^2+\cdots+|z_n|^2}$.
\begin{prop}[Lemma 1.2 in \cite{Tian}]\label{proposition: peak-section}
	For an $n$-tuple of integers $p=(p_1,\cdots,p_n)\in\mathbb{Z}_+^n$ and an integer $r>|p|=p_1+\cdots+p_n$, there exists an $m_0>0$ such that, for $m>m_0$, there is a holomorphic global section $S$, called a {\em peak section}, of the line bundle $L^{\otimes m}$, satisfying 
	\begin{equation}\label{equation: property-peak-section-norm}
 		\int_X ||S||^2_{h^m}dV_g=1,\hspace{5mm}\int_{X\setminus\{\rho(z)\leq\frac{\log m}{\sqrt{m}}\}}||S||_{h^m}^2dV_g=O\left(\frac{1}{m^{2r}}\right),
	\end{equation}
	and locally at $z_0$, 
	\begin{equation}\label{equation: Taylor-expansion-peak-section}
 		S(z)=\lambda_{m,p}\cdot\left(z_1^{p_1}\cdots z_n^{p_n}+O(|z|^{2r})\right)e_{L,z_0}^m\left(1+O\left(\frac{1}{m^{2r}}\right)\right),
	\end{equation}
	where $||\cdot||_{h^m}$ is the norm on $L^{\otimes m}$ given by $h^m$, and $O\left(\frac{1}{m^{2r}}\right)$ denotes a quantity dominated by $C/m^{2r}$ with the constant $C$ depending only on $r$ and the geometry of $X$, moreover
	\begin{equation}\label{equation: coefficient-leading-term-Taylor-expansion-peak-section}
 		\lambda^{-2}_{m,p}=\int_{\rho(z)\leq\log m/\sqrt{m}}|z_1^{p_1}\cdots z_n^{p_n}|^2\cdot e^{-m\cdot\varphi(z)} dV_g,
	\end{equation}
	where $dV_g=\det(g_{i\bar{j}})(\sqrt{-1}/(2\pi))^ndz_1\wedge d\bar{z}_1\wedge\cdots\wedge dz_n\wedge d\bar{z}_n$ is the volume form. 
\end{prop}
Here we use the same notation as in the introduction of $K$-frame: $h^m(e_{L,z_0},e_{L,z_0})=e^{-m\cdot\varphi(z)}.$ Geometrically, a peak section is, roughly speaking, a global holomorphic section of a high enough tensor power of $L$ whose norm is $1$ and is concentrated around a given point $z_0$ on the K\"ahler manifold. 


We want to define a section $S_{m,p,r}$ of the line bundle $L^{\otimes m}$ by normalizing the peak section $S(z)$ in Proposition \ref{proposition: peak-section} so that its Taylor expansion at $z_0$ under the K-frame $e_{L,z_0}^{\otimes m}$ is exactly equal to $z_1^{p_1}\cdots z_{n}^{p_n}$ up to order $2r-1$. This forces $S_{m,p,r}$ to be of the form:
\begin{equation}\label{equation: normalized-peak-section}
 S_{m,p,r}:=\lambda^{-1}_{m,p}\cdot \left(1+O\left(\frac{1}{m^{2r}}\right)\right)\cdot S(z).
\end{equation}
We now give an estimate of $\lambda^{-1}_{m,p}$. We have, for $m>>0$,
\begin{align*}
\lambda^{-2}_{m,p}
& = \int_{\rho(z)\leq\log m/\sqrt{m}}|z_1^{p_1}\cdots z_n^{p_n}|^2\cdot e^{-m\cdot\rho(z)}dV_g\\
& \leq \int_{\rho(z)\leq 1}|z_1^{p_1}\cdots z_n^{p_n}|^2\cdot e^{-m\cdot\rho(z)}dV_g=O\left(\frac{1}{m^{|p|+n}}\right),
\end{align*}
where the estimate follows from Theorem \ref{theorem: Feynman-Laplace-Gaussian-integral}. In particular, there is a constant $C_{p}$, depending only on the point $z_0$ and the multi-index $p$, such that
$$
\lambda^{-1}_{m,p}\cdot \left(1+O\left(\frac{1}{m^{2r}}\right)\right)\leq C_{p}\cdot\left(\frac{1}{m^{\frac{|p|+n}{2}}}\right).
$$
We define a normalized version of the inner product of sections of $L^{\otimes m}$: 
\begin{defn}
	Let $s_1, s_2$ be (smooth) sections of $L^{\otimes m}$. Their (normalized) inner product is defined as
	\begin{equation}\label{equation: normalized-inner-product-peak-sections}
		\langle s_1,s_2\rangle_{m}:=m^{n}\cdot\int_X h^m(s_1,s_2)dV_g,
 	\end{equation}
	where $n = \dim_\mathbb{C} X$, and we let $||s||_m$ be the norm of a section $s$ under this inner product. 
\end{defn}
\begin{rmk}
	An explanation of the normalization factor $m^n$ is the following: consider $\mathbb{C}^n$ with the volume form 
	$$
	\left(\frac{\sqrt{-1}}{2\pi}\right)^n e^{-m\cdot|z|^2}dz_1\wedge d\bar{z}_1\wedge\cdots\wedge dz_n\wedge d\bar{z}_n,
	$$
	then the factor $m^n$ normalizes the volume to $1$ under this volume form. 
\end{rmk}

We summarize the properties of  $S_{m,p,r}$ as follows:
\begin{equation}\label{equation: peak-section-norm-bounded}
||S_{m,p,r}||_m^2\leq m^n\cdot C\cdot\lambda_{m,p}^{-2}=O\left(\frac{1}{m^{|p|}}\right);
\end{equation}
\begin{equation}\label{equation: norm-peak-section-outside-disk}
\begin{aligned}
m^n\cdot\int_{M\setminus\{\rho(z)\leq1\}}||S_{m,p,r}||_{h^m}^2dV_g&\leq m^n\cdot\int_{M\setminus\{\rho(z)\leq\log m/\sqrt{m}\}}||S_{m,p,r}||_{h^m}^2dV_g\\
&\leq m^n\cdot C\cdot\lambda^{-2}_{m,p}\cdot O\left(\frac{1}{m^{2r}}\right) =O\left(\frac{1}{m^{2r+|p|}}\right).
\end{aligned}
\end{equation}
\begin{rmk}
	The constant $C$ in the above estimates comes from the number $1+O\left(\frac{1}{m^{2r}}\right)$ in equation \eqref{equation: normalized-peak-section}. 
\end{rmk}
Locally around $z_0$, we have
\begin{equation}\label{equation: peak-section-leading-term}
S_{m,p,r}(z)=\left(z_1^{p_1}\cdots z_n^{p_n}+O(|z|^{2r})\right)\cdot e_{L}^m.
\end{equation}

The first property \eqref{equation: peak-section-norm-bounded} implies that for fixed $p,r$, the sequence $\{S_{m,p,r}\}$ is bounded for all $m$. 
The second property \eqref{equation: norm-peak-section-outside-disk} is saying that the sections $S_{m,p,r}$ are asymptotically ``concentrated'' around the point $z_0$.
The third property \eqref{equation: peak-section-leading-term} is saying that asymptotically, $S_{m,p,r}$ has an assigned leading term of the Taylor expansion at the point $z_0$. 
\begin{rmk}
 	The third property of $S_{m,p,r}$ is the reason for calling it a {\em normalized peak section}:
 	its Taylor expansion at $z_0$ has leading term exactly exactly equal to the monomial $z^p\cdot e_{L}^m$ corresponding to the multi-index $p$.
\end{rmk}
\begin{rmk}
	According to \cite{Tian}, for every fixed $p,r$, peak sections exist only when $m$ is big enough. We will adopt the follows two conventions 
	\begin{itemize}
	 \item $S_{m,p,r}:=0$ for small $m$,
	 \item $S_{m,p,r}:=0$ if $r\leq |p|$.
	\end{itemize}

\end{rmk}

There is the following estimate of the inner product between peak sections:
\begin{lem}\label{lemma: inner-product-peak-section-asymptotic}
	Given two normalized peak sections $S_{m,p^1,r}, S_{m,p^2,r}$,
	we have the following asymptotic expansion of their inner product up to order $r$:
	\begin{equation}\label{equation: asymptotic-expansion-inner-product-peak-section}
	\langle S_{m,p^1,r}, S_{m,p^2,r} \rangle_m-\sum_{k=1}^{p'-1}a_k\cdot\frac{1}{m^k} = O\left(\frac{1}{m^{p'}}\right),
	\end{equation}
	where the coefficients $a_k$'s are the same as those in the expansion of the following formal integral:
	$$
	\int z^{p^1}\bar{z}^{p^2}e^{\frac{-|z|^2+\phi(z,\bar{z})}{\hbar}}=\sum_{k\geq 1}a_k\cdot\hbar^k.
	$$
	Thus, for fixed multi-indices $p^1,p^2$, the $a_i$'s are independent of $r>>0$. In particular, the leading term of the asymptotic expansion of $||S_{m,p,r}||_m^2$ is given by
	$$
	\left(\frac{1}{m}\right)^{|p|}p!,
	$$
	which is non-zero. 
\end{lem}
\begin{proof}
We split the integral defining the inner product to two parts:
$$
\frac{1}{m^n}\int_{\{\rho(z)<1\}} h^m(S_{m,p^1,r}, S_{m,p^2,r})\cdot dV_{g}+\frac{1}{m^n}\int_{X\setminus\{\rho(z)<1\}} h^m(S_{m,p^1,r}, S_{m,p^2,r})\cdot dV_{g},
$$
where the second part is $O\left(\frac{1}{m^{2r+\frac{|p^1|}{2}+\frac{|p^2|}{2}}}\right)$ by using Cauchy-Schwarz inequality and equation \eqref{equation: norm-peak-section-outside-disk}. Thus to show equation \eqref{equation: asymptotic-expansion-inner-product-peak-section},  the integral outside the disk $\{\rho(z)<1\}$ can be ignored. For the integral over the disk, we can apply Theorem \ref{theorem: Feynman-Laplace-Gaussian-integral} to obtain the desired asymptotic expansion. In particular, the coefficients $a_k$'s are the same as those coming from the formal integral.
\end{proof}

As an immediate corollary, we have the following:
\begin{cor}
Let $p^1, p^2$ be multi-indices, and let $r>\max\{|p^1|,|p^2|\}$, then we have the following estimate of the inner product between $S_{m,p^1,r}$ and $S_{m,p^2,r}$:
$$\langle S_{m,p^1,r}, S_{m,p^2,r}\rangle_m = 
\begin{cases}
O\left(\frac{1}{m^{|p^1|}}\right),& p^1=p^2\\
o\left(\frac{1}{m^{\max\{|p^1|,|p^2|\}}}\right),& p^1 \neq p^2.
\end{cases}$$
\end{cor}
\begin{proof}
 	The case where $p_1=p_2$ is given by equation \eqref{equation: peak-section-norm-bounded}. For $p_1\not=p_2$, we need to estimate an integral. For the integral inside the disk $\{\rho(z)<1\}$, this estimate is given by Corollary \ref{corollary: leading-term-inner-product-polynomial} where the technical condition on $\phi$ is implied by the existence of $K$-coordinates and $K$-frame. For the estimate of the integral outside the disk, we use Cauchy-Schwarz inequality:
 	\begin{equation}\label{equation: estimiate-integral-outside-disk}
 		m^n\cdot\int_{X\setminus\{\rho(z)<1\}}h^m(S_{m,p_1,r}, S_{m,p_2,r})\cdot dV_g=O\left(\frac{1}{m^{\frac{4r+|p_1|+|p_2|}{2}}}\right)=o\left(\frac{1}{m^{\max\{|p_1|,|p_2|\}}}\right).
 	\end{equation}
\end{proof}

\end{document}